\pdfoutput=1
\RequirePackage{ifpdf}
\ifpdf
\documentclass[pdftex]{sigma}
\else
\documentclass{sigma}
\fi

\numberwithin{equation}{section}

\newtheorem{Theorem}{Theorem}[section]
\newtheorem{Corollary}[Theorem]{Corollary}
\newtheorem{Lemma}[Theorem]{Lemma}
\newtheorem{Proposition}[Theorem]{Proposition}
 { \theoremstyle{definition}
\newtheorem{Definition}[Theorem]{Definition}
\newtheorem{Example}[Theorem]{Example}
\newtheorem{Remark}[Theorem]{Remark} }

\newcommand{\xtil}{\widetilde{X}}

\newcommand{\ytil}{\widetilde{Y}}
\newcommand{\lra}{\longrightarrow}
\newcommand{\lla}{\longleftarrow}
\newcommand{\PP}{\mathbb P}
\newcommand{\QQ}{\mathbb Q}
\newcommand{\ZZ}{\mathbb Z}

\newcommand{\ql}{{\QQ}_{\ell}}
\newcommand{\ms}{{\mathbb G}_m}
\newcommand{\h}{\operatorname{ht}}
\newcommand{\bgam}{\boldsymbol{\gamma}}

\newcommand{\co}{\mathcal{O}}

\begin{document}

\newcommand{\arXivNumber}{1805.04233}

\renewcommand{\thefootnote}{}

\renewcommand{\PaperNumber}{097}

\FirstPageHeading

\ShortArticleName{A Note on the Formal Groups of Weighted Delsarte Threefolds}

\ArticleName{A Note on the Formal Groups\\
of Weighted Delsarte Threefolds\footnote{This paper is a~contribution to the Special Issue on Modular Forms and String Theory in honor of Noriko Yui. The full collection is available at \href{http://www.emis.de/journals/SIGMA/modular-forms.html}{http://www.emis.de/journals/SIGMA/modular-forms.html}}}

\Author{Yasuhiro GOTO}

\AuthorNameForHeading{Y.~Goto}

\Address{Department of Mathematics, Hokkaido University of Education,\\
1-2 Hachiman-cho, Hakodate 040-8567 Japan}
\Email{\href{mailto:goto.yasuhiro@h.hokkyodai.ac.jp}{goto.yasuhiro@h.hokkyodai.ac.jp}}

\ArticleDates{Received May 11, 2018, in final form August 30, 2018; Published online September 12, 2018}

\Abstract{One-dimensional formal groups over an algebraically closed field of positive cha\-rac\-teristic are classified by their height. In the case of $K3$ surfaces, the height of their formal groups takes integer values between~$1$ and~$10$, or~$\infty$. For Calabi--Yau threefolds, the height is bounded by~$h^{1,2}+1$ if it is finite, where~$h^{1,2}$ is a Hodge number. At present, there are only a limited number of concrete examples for explicit values or the distribution of the height. In this paper, we consider Calabi--Yau threefolds arising from weighted Delsarte threefolds in positive characteristic. We describe an algorithm for computing the height of their formal groups and carry out calculations with various Calabi--Yau threefolds of Delsarte type.}


\Classification{14L05; 14J32}

\begin{flushright}
\begin{minipage}{80mm}
\it Dedicated to Professor Noriko Yui, from whom I learned, as a graduate student and a collaborator now, enthusiasm, dynamics and humanity in mathematical research.
\end{minipage}
\end{flushright}

\renewcommand{\thefootnote}{\arabic{footnote}}
\setcounter{footnote}{0}

\section{Introduction}

Let $k$ be an algebraically closed field of characteristic $p >0$. Let $X$ be a Calabi--Yau threefold over $k$, by which we mean a~smooth projective variety over $k$ of dimension $3$ with a trivial canonical sheaf and $\dim H^1(X, \co_X )=\dim H^2(X, \co_X )=0$. In \cite{AM}, Artin and Mazur defined a~functor~$\Phi_X$ for $X$ on the category of finite local $k$-algebras $A$ with residue field $k$ by
\begin{gather*}
\Phi_X (A)=\ker \big(H_{\rm et}^3(X_A,\ms )\lra H_{\rm et}^3(X,\ms )\big),
\end{gather*}
where $X_A=X\times \operatorname{Spec} A$ and $\ms$ is the sheaf of multiplicative groups. It is proved in \cite{AM} that this functor (more generally, those for Calabi--Yau varieties of any dimension) is representable by a smooth formal group of dimension equal to the geometric genus of~$X$, which is one in our case. By abuse of notation, we also use $\Phi_X$ for this formal group and called it the {\it $($Artin--Mazur$)$ formal group} of $X$.

A formal group in positive characteristic $p$ is endowed with the multiplication-by-$p$ map. The $p$-rank of its kernel is called the {\it height} of $\Phi_X$ and denoted by $h:=\h \Phi_X$, namely
\begin{gather*}
p^h= \# \ker ([p]\colon \Phi_X \lra \Phi_X ).
\end{gather*}
It is known that one-dimensional formal groups in positive characteristic are determined up to isomorphism by the height. In the case of $K3$ surfaces, the height takes integer values between~$1$ and~$10$, or~$\infty$ (cf.~\cite{Ar,G1,Y1}). In the case of Calabi--Yau threefolds, it is proved in~\cite{GK} that $h$ is bounded above by~$h^{1,2}+1$ if $h\neq \infty$, where $h^{1,2}$ is a~Hodge number of~$X$. Note that it is still a~conjecture that~$h^{1,2}$ is bounded for Calabi--Yau threefolds.

In this paper, we consider Calabi--Yau threefolds arising from weighted projective hypersurfaces of Delsarte type and give an algorithm to compute the height of their Artin--Mazur formal groups. The algorithm involves only combinatorial argument so that one can actually do calculations to find numerical data for the height of Calabi--Yau threefolds. Among others, we find a~Calabi--Yau threefold with height $42$ or $82$ in some characteristic.

\section{Quotient maps and the height of formal groups}\label{section2}

In this section, we show that the formal groups are invariant under the quotient map by a~group of {\it symplectic} automorphisms (by which we mean the automorphisms that preserve the dualizing sheaf of a variety). The Calabi--Yau threefolds we consider in Sections~\ref{section4} and~\ref{section5} often have large groups of symplectic automorphisms. In this section, we focus our attention to a very special case, namely, orbifold Calabi--Yau threefolds $(X,Y)$ which form a~mirror pair (i.e., $h^{1,1}(X)=h^{2,1}(Y)$ and $h^{2,1}(X) =h^{1,1}(Y)$), and discuss their formal groups briefly. More details about relationships between quotient maps and formal groups will be discussed elsewhere.

\begin{Lemma} \label{mirror1} Let $X$ be a Calabi--Yau threefold over $k$ and $G$ be a finite group acting on $X$ symplectically $($i.e., preserving the dualizing sheaf of~$X)$. Write $Y:=X/G$. Assume that~$p$ is coprime to the order of~$G$ and that there exists a~crepant resolution~$\ytil$ of~$Y$. Then~$\ytil$ is a~Calabi--Yau threefold with $\Phi_X \cong \Phi_{\ytil}$. In particular, the formal groups of $X$ and $\ytil$ have the same height: $\h \Phi_X = \h \Phi_{\ytil}$.
\end{Lemma}

\begin{proof}Since $p$ is coprime to the order of $G$, $Y$ has at most rational singularities and by Theorem~3.1 (and the paragraph~3.10) of~\cite{St}, we have $\Phi_Y \cong \Phi_{\ytil}$ for a crepant resolution~$\ytil$ of~$Y$. Write $f\colon X \lra Y$ for the quotient map. As $G$ acts on $X$ symplectically, $g^{*} \omega_X\cong \omega_X$ for every $g\in G$ and thus $\omega_X^G \cong \omega_X$ for the dualizing sheaf $\omega_X$ of $X$. Since $X$ is Calabi--Yau, we find $\co_Y \cong f_{*}\co_X^G \cong f_{*} \omega_X^G \cong f_{*} \omega_X \cong f_{*}\co_X$. Again by Theorem 3.1 of \cite{St}, we see $\Phi_X \cong \Phi_Y$. Hence $\Phi_X \cong \Phi_{\ytil}$.
\end{proof}

\begin{Proposition} \label{mirror2}Let $X$ be a threefold over $k$ with $\co_X \cong \omega_X$ and $G$ be a~finite group acting on~$X$ symplectically. Write $\xtil$ and~$\ytil$ for crepant resolutions of~$X$ and~$Y$, respectively, constructed as in the following diagram:
\begin{gather} \label{dgm-1}
\begin{array}{@{}clclc}
X& & \lla & \xtil \\
\downarrow & & & & \\
X/G & =Y & \lla & \ytil.
\end{array}
\end{gather}
Write $h:= \h \Phi_{\xtil}$. Assume that $\big(\xtil ,\ytil \big)$ is a mirror pair of Calabi--Yau threefolds and $h$ is finite. Then
\begin{gather*}
h \leq \min \big\{ h^{1,1}\big(\xtil\big)+1,\, h^{1,2}\big(\xtil\big)+1 \big\}.
\end{gather*}
\end{Proposition}

\begin{proof}Note that $h= \h \Phi_{\xtil} =\h \Phi_{\ytil}$ by Lemma~\ref{mirror1}. Since $h$ is finite, it follows from Corollary~2.3 of~\cite{GK} that
\begin{gather*}
h \leq h^{1,2}\big(\xtil\big)+1 \qquad \text{and} \qquad h \leq h^{1,2}\big(\ytil\big)+1.
\end{gather*}
As $\big(\xtil,\ytil\big)$ is a mirror pair of Calabi--Yau threefolds, we find $h^{1,2}\big(\ytil\big)+1=h^{1,1}\big(\xtil\big)+1$. Combining these relations, we obtain the asserted inequality.
\end{proof}

\begin{Remark}It is still an open problem whether or not $h^{1,1}$ or $h^{1,2}$ is bounded for Calabi--Yau threefolds.
\end{Remark}

In later sections of this paper, we see examples of Calabi--Yau threefolds with many symplectic automorphisms. One may then expect to construct their mirror partners by using some group actions~$G$ as in Proposition~\ref{mirror2}. In some cases, however, we do not find such $G$ as described in the following statement; a concrete example for this can be found in Example~\ref{ex6-3}.

\begin{Corollary} \label{mirror3} Let $\xtil$ be a Calabi--Yau threefold defined as a crepant resolution of $X$ and write $h= \h \Phi_{\xtil}$. Assume that $h > \min \big\{ h^{1,1}\big(\xtil\big)+1,\,
h^{1,2}\big(\xtil\big)+1 \big\}$. Then either $h=\infty$ or there exists no group $G$ of symplectic automorphisms on $X$ such that $\big(\xtil,\ytil\big)$ of Proposition~{\rm \ref{mirror2}} forms a mirror pair.
\end{Corollary}

\begin{proof}Take the contrapositive of Proposition \ref{mirror2}.
\end{proof}

\section{Weighted Delsarte varieties}\label{section3}

In order to compute the cohomology groups of weighted Delsarte varieties, we explain some geometric properties of them (see also~\cite{G}).

Let $Q=(q_0,\dots,q_n)$ be an $(n+1)$-tuple of positive integers such that $p\nmid q_i$, $0\leq i\leq n$, and $\gcd (q_0,\dots , \hat{q_i},\dots ,q_n)=1$ for every $0\leq i \leq n$, where $\hat{q_i}$ means that $q_i$ is omitted. The weighted projective $n$-space over $k$ of type $Q$, denoted by $\PP^n(Q)$, is the projective variety $\PP^n(Q):=\operatorname{Proj} k[x_0,\dots ,x_n]$, where the polynomial algebra is graded by the condition $\deg (x_i)=q_i$ for $0\leq i\leq n$ (cf.~\cite{Dol}).

Let $m$ be a positive integer such that $p\nmid m$. Let $A= (a_{ij})$ be an $(n+1)\times (n+1)$ matrix of integer entries having the properties
\begin{enumerate}\itemsep=0pt
\item[(i)] $a_{ij}\geq 0$ and $ p\nmid a_{ij}$ for every $(i,j)$ with $a_{ij}\neq 0$,
\item[(ii)] $p\nmid \det A$,
\item[(iii)] $\sum\limits_{j=0}^{n} q_ja_{ij} =m$ for $0\leq i\leq n$,
\item[(iv)] given $j$, $a_{ij}=0$ for some $i$.
\end{enumerate}
We define an $(n-1)$-dimensional {\it weighted Delsarte variety in $\PP^n(Q)$ of degree $m$ with matrix~$A$} (cf.~\cite{Del,G,S}) to be the weighted projective hypersurface,~$X_A$, defined by
\begin{gather*}
\sum_{i=0}^{n} x_0^{a_{i0}}x_1^{a_{i1}}\cdots x_n^{a_{in}} =0 \subset \PP^n(Q) .
\end{gather*}
When $A$ is a diagonal matrix, the equation has the form
\begin{gather*}
x_0^{d_0} +x_1^{d_1} + \cdots +x_n^{d_n} =0
\end{gather*}
and we call it a {\it weighted Fermat variety of degree~$d$}.

Weighted Delsarte varieties are birational to finite quotients of Fermat varieties and many properties of their cohomology groups can be extracted from those of Fermat varieties (cf.~\cite{G,SK,Y2}). For instance, write $d=|\det A|$ and let $F_d$ be the $(n-1)$-dimensional Fermat variety of degree~$d$ in the usual projective space $\PP^n$:
\begin{gather*}
F_d\colon \ y_0^d +y_1^d + \cdots +y_n^d =0.
\end{gather*}
Write $\mu_d$ for the group of $d$-th roots of unity in $k^{\times}$, which consists of $d$ elements as $p\nmid d$. Set $\Gamma =(\mu_d \times \cdots \times \mu_d )/(\text{diagonal elements})$ for the product of~$n+1$ copies of $\mu_d$ modulo diagonal elements and define
\begin{gather} \label{ga}
{\Gamma}_A=\left\{ \bgam =\left(\prod_{j=0}^n \gamma_j^{a_{0j}}, \prod_{j=0}^n \gamma_j^{a_{1j}},\dots ,\prod_{j=0}^n \gamma_j^{a_{nj}} \right) \in \Gamma \,\Bigg|\, (\gamma_0,\dots ,\gamma_n )\in \mu_d^{n+1} \right\} .
\end{gather}
Then ${\Gamma}_A$ is a subgroup of $\Gamma$ and it acts on~$F_d$ by
\begin{gather*}
\bgam \cdot (y_0:y_1:\cdots :y_n)
=\left( \left( \prod_{j=0}^n \gamma_j^{a_{0j}} \right)y_0: \left( \prod_{j=0}^n \gamma_j^{a_{1j}} \right)y_1:\cdots : \left( \prod_{j=0}^n \gamma_j^{a_{nj}} \right)y_n\right)
\end{gather*}
for $\bgam \in {\Gamma}_A$ and $(y_0:y_1:\cdots :y_n)\in F_d$.

\begin{Lemma} \label{quot}Let $X_A$ be a weighted Delsarte variety in $\PP^n(Q)$ with matrix~$A$. Then $X_A$ is birational to the quotient $F_d/\Gamma_A$.
\end{Lemma}

\begin{proof} Write $B=(b_{ij})$ for the cofactor matrix of $A$. Then there is a dominant rational map $f\colon F_d \cdots \rightarrow X_A$ defined by
\begin{gather*}
f((y_0:y_1:\cdots :y_n))=\left( \prod_{j=0}^n y_j^{b_{0j}}: \prod_{j=0}^n y_j^{b_{1j}}:\cdots : \prod_{j=0}^n y_j^{b_{nj}} \right) .
\end{gather*}
Hence $X_A$ is birational to $F_d/\Gamma_A$.
\end{proof}

We describe the $\ell$-adic \'etale cohomology of the varieties involved, where $\ell$ is a prime different from $p=\operatorname{char} k$. It is known that the cohomology of Fermat variety $F_d$ is decomposed into one-dimensional pieces parameterized by the characters of~$\Gamma$. In fact, note first that $F_d$ is of dimension $n-1$ and
\begin{gather*}
H^{n-1}_{\rm et}(F_d,\QQ_{\ell})\cong
\begin{cases}
H^{n-1}_{\rm prim}(F_d,\QQ_{\ell}) & \text{if $n$ is even}, \\
V(0) \oplus H^{n-1}_{\rm prim}(F_d,\QQ_{\ell}) & \text{if $n$ is odd},
\end{cases}
\end{gather*}
where $V(0)$ denotes the subspace corresponding to the hyperplane section and $H^{n-1}_{\rm prim}(F_d,\QQ_{\ell})$ is the primitive part of $H^{n-1}(F_d,\QQ_{\ell})$ (cf.~\cite{Ka, Sh}). Define
\begin{gather*}
\mathfrak{A} (F_d)=\left\{ \boldsymbol{\alpha} =(\alpha_0 ,\alpha_1 ,\dots ,\alpha_n )\,\Bigg\vert\, \alpha_i \in \ZZ /d\ZZ,\, \alpha_i \neq 0, \ 0\leq i \leq n,\ \sum_{i=0}^n \alpha_i =0 \right\}
\end{gather*}
and
\begin{gather*}
V(\boldsymbol{\alpha} )=\big\{ v\in H^{n-1}_{\rm prim}(F_d,\ql )\,|\, \bgam^{*} (v)= \gamma_0^{\alpha_0} \gamma_1^{\alpha_1} \cdots \gamma_n^{\alpha_n} \cdot v ,\ \forall\, \bgam =(\gamma_0,\dots ,\gamma_n) \in \Gamma \big\},
\end{gather*}
where $\bgam^{*}$ is the endomorphism of $H^{n-1}_{\rm et}(F_d,\ql )$ induced from $\bgam$. Then the primitive part $H^{n-1}_{\rm prim}(F_d,\ql )$ of the cohomology for $F_d$ is given as
\begin{gather*}
H^{n-1}_{\rm prim}(F_d,\ql )=\bigoplus_{\boldsymbol{\alpha} \in \mathfrak{A} (F_d )} V(\boldsymbol{\alpha} ).
\end{gather*}

A similar property holds for $X_A$ since it is birational to the quotient variety $F_d/\Gamma_A$. Here we describe the cohomology of $F_d/\Gamma_A$.

\begin{Lemma} \label{coh-wdv} Let $X_A$ be a weighted Delsarte variety in $\PP^n(Q)$ with matrix~$A$. Let~$\Gamma_A$ be the group defined in \eqref{ga} and put
\begin{gather*}
\mathfrak{A} (X_A)=\left\{ \boldsymbol{\alpha} =(\alpha_0 ,\alpha_1 ,\dots ,\alpha_n )\in \mathfrak{A} (F_d) \,\Bigg\vert\, \sum_{i=0}^n a_{ij} \alpha_i =0, \ 0\leq j\leq n \right\}.
\end{gather*}
Then
\begin{gather*}
H^{n-1}_{\rm et}(F_d/\Gamma_A ,\ql )\cong
\begin{cases}
\bigoplus_{\boldsymbol{\alpha} \in \mathfrak{A} (X_A )}V(\boldsymbol{\alpha} ) & \text{if $n$ is even}, \\
V(0) \oplus \bigoplus_{\boldsymbol{\alpha} \in \mathfrak{A} (X_A )}V(\boldsymbol{\alpha} ) & \text{if $n$ is odd}.
\end{cases}
\end{gather*}
\end{Lemma}

\begin{proof}The proof is similar to the case of weighted Delsarte (or diagonal) surfaces; see, for instance, \cite{Gt2} and \cite{G} for details. The assertion follows from the isomorphism
\begin{gather*}
H^{n-1}_{\rm et} (F_d/\Gamma_A ,\ql )\cong H^{n-1}_{\rm et}(F_d, \ql )^{\Gamma_A}
\end{gather*}
as $\ql$-vector spaces and by considering the $\Gamma_A$-action on $H^{n-1}_{\rm et}(F_d,\ql )$.
\end{proof}

It is a bit lengthy to describe the set $\mathfrak{A} (X_A )$ etc. here. The reader may refer, for instance, to~\cite{Gt2} or~\cite{Y1} for some concrete examples of $\mathfrak{A} (X_A )$.

\section{Calabi--Yau threefolds of Delsarte type}\label{section4}

In this section, we discuss Calabi--Yau threefolds arising from weighted Delsarte threefolds. Using their cohomology groups, we describe an algorithm for computing the height of their formal groups.

Let $X$ be a weighted projective variety in $\PP^n(Q)$ with $Q= (q_0,\dots ,q_n)$. $X$ is said to be {\it quasi-smooth} (cf.~\cite{Dol}) if its affine quasi-cone is smooth outside the origin. For instance, weighted Fermat varieties are quasi-smooth. As a~special case of weighted quasi-smooth varieties, we observe the following property for quasi-smooth weighted Delsarte varieties.

\begin{Lemma} \label{sing-wd} Let $X_A$ be a quasi-smooth weighted Delsarte variety in $\PP^n(Q)$ of degree $m$ with matrix $A$. Write $d=|\det A|$. Let $\Gamma_A$ be the group defined in $(\ref{ga})$ acting on the
Fermat variety~$F_d$. Then the following assertions hold:
\begin{enumerate}\itemsep=0pt
\item[$(1)$] The quotient variety $F_d/\Gamma_A$ has at most rational abelian quotient singularities.
\item[$(2)$] $X_A$ has at most rational cyclic quotient singularities.
\end{enumerate}
\end{Lemma}

\begin{proof} (1) Since $F_d$ is a smooth variety and $\Gamma_A$ is an abelian group, $F_d/\Gamma_A$ has at most abelian quotient singularities. In characteristic~0, quotient singularities are known to be rational. Hence we only need to show that $F_d/\Gamma_A$ is liftable to characteristic 0 and this follows from conditions~(i) and~(ii) on matrix~$A$.

(2) A quasi-smooth variety is locally isomorphic to the quotient of a smooth variety by some cyclic group action and this cyclic group is a~subgroup of $\mu_{q_i}$ for some weight~$q_i$ (cf.~\cite{Dol}). Since~$p\nmid q_i$ for every~$i$, the group action by a subgroup of $\mu_{q_i}$ can be lifted to characteristic 0. Hence~$X_A$ has at most cyclic quotient singularities and they are rational.
\end{proof}

Now we consider weighted Delsarte threefolds.

\begin{Lemma} \label{cy} Let $X_A$ be a weighted Delsarte threefold in $\PP^4(Q)$ with matrix $A$. Assume that~$X_A$ is quasi-smooth and $m=q_0+q_1+q_2+q_3+q_4$. Then the dualizing sheaf of $X_A$ is trivial and there exists a crepant resolution for $X_A$.
\end{Lemma}

\begin{proof}Since $X_A$ is a quasi-smooth hypersurface and of dimension~3, it is known (cf.~\cite[Proposition~6]{Dim}) to be in general position relative to $\PP^4(Q)_{\rm sing}$ (i.e., $\operatorname{codim}_{X_A} \big(X_A\cap \PP^4(Q)_{\rm sing}\big)\geq 2$, where $\PP^4(Q)_{\rm sing}$ is the singular locus of $\PP^4(Q)$). Hence the dualizing sheaf of~$X_A$ is computed as~$\omega_{X_A} \cong \co_{X_A} (m-q_0-q_1-q_2-q_3-q_4)\cong \co_{X_A}$ (see~\cite{Dol}). The existence of a crepant resolution for $X_A$ is proved in~\cite{GRY}.
\end{proof}

\begin{Definition}If $X_A$ is a quasi-smooth weighted Delsarte threefold with matrix~$A$ of degree~$m$ with $m=q_0+ \cdots +q_4$, then a crepant resolution $\xtil_A$ of $X_A$ is called a {\it Calabi--Yau threefold of $($weighted$)$ Delsarte type in~$\PP^4(Q)$ with matrix~$A$}. When $A$ is a diagonal matrix, $X_A$ is also called a {\it Calabi--Yau threefold of $($weighted$)$ Fermat type}.
\end{Definition}

Since the quotient $F_d/\Gamma_A$ is birational to $X_A$, it is also birational to $\xtil_A$. Using the cohomological information of $F_d/\Gamma_A$, we write several birational properties of $\xtil_A$. Recall that $A=(a_{ij})$, $d=|\det A|$ and that $F_d$ is the Fermat threefold of degree $d$ in $\PP^4$. We have
\begin{gather*}
H^3_{\rm et}(F_d/\Gamma_A ,\ql )\cong \bigoplus_{\boldsymbol{\alpha} \in \mathfrak{A} (X_A )} V(\boldsymbol{\alpha} ),
\end{gather*}
where
\begin{gather*}
\mathfrak{A} (X_A)=\left\{ (\alpha_0 ,\dots ,\alpha_4 )\in (\ZZ /d\ZZ )^5
\left|
\begin{matrix}
\alpha_i \neq 0, \ 0\leq i\leq 4,\
\sum\limits_{i=0}^4\alpha_i =0 \\
\sum\limits_{i=0}^4 a_{ij} \alpha_i =0 \ \text{for } 0\leq j\leq 4
\end{matrix}
\right. \right\}.
\end{gather*}
For each $\boldsymbol{\alpha} =(\alpha_0 ,\dots ,\alpha_4 )\in \mathfrak{A} (X_A)$,
define an integer
\begin{gather*}
\| \boldsymbol{\alpha} \| =\sum_{i=0}^{4} \left\langle\frac{\alpha_i}{d} \right\rangle -1,
\end{gather*}
where $<\alpha_i /d>$ denotes the fractional part of $\alpha_i /d$. It takes values $\| \boldsymbol{\alpha} \| =0,1,2$ or $3$. If $\xtil_A$ is a Calabi--Yau threefold of Delsarte type in $\PP^4(Q)$ with matrix $A$, then there exists a unique element
\begin{gather*}
\boldsymbol{\alpha}_0=(\alpha_0 , \alpha_1 ,\dots ,\alpha_4 )\in \mathfrak{A} (X_A)
\end{gather*}
with $\| \boldsymbol{\alpha}_0 \| =0$ (cf.~\cite{SY,Y1}). Note that $\| -\boldsymbol{\alpha}_0 \| =(4-1)-\| \boldsymbol{\alpha}_0 \|=3-0=3$.

Recall that $p\ (=\operatorname{char} k)$ is relatively prime to $d$. Let $f$ be the order of~$p$ modulo~$d$. Put
\begin{gather*}
H= \big\{p^i\ (\operatorname{mod} d)\,|\, 0\leq i< f \big\},
\end{gather*}
which is a subgroup of $(\ZZ /d\ZZ )^{\times}$. For $\boldsymbol{\alpha} =(\alpha_0 , \dots ,\alpha_4 )\in \mathfrak{A} (X_A)$, we define a non-negative integer
\begin{gather*}
\mathcal{A}_H(\boldsymbol{\alpha} )=\sum_{t\in H} \| t\boldsymbol{\alpha} \|.
\end{gather*}
It is known that the part (line segments) of the Newton polygon of $\Phi_{\xtil_A}$ with slope less than $1$ corresponds to the height of the formal group of~$\xtil_A$; if there is no such part, then $\h \Phi_{\xtil_A} =\infty$ (cf.~\cite{AM,GK}). In our situation, $X_A$, $\xtil_A$ and $F_d/\Gamma_A$ are birational to each other (by resolving rational singularities) and the formal groups are invariant under resolution of rational singularities; see Lemma~\ref{mirror1}. Hence the height of~$\Phi_{\xtil_A}$ can be computed from the Newton polygon of~$F_d/\Gamma_A$. The slopes of this polygon is given by $\mathcal{A}_H(\boldsymbol{\alpha} )/f$ and the length of the part with slope less than~$1$ is equal to the number of~$\boldsymbol{\alpha} \in \mathfrak{A} (X_A)$ satisfying $\mathcal{A}_H(\boldsymbol{\alpha} )/f <1$. In summary, the height of~$\Phi_{\xtil_A}$ can be calculated by evaluating $\mathcal{A}_H(\boldsymbol{\alpha} )$ for~$\boldsymbol{\alpha} \in \mathfrak{A} (X_A)$.

\begin{Lemma} \label{lambda-0} Let $\xtil_A$ be a Calabi--Yau threefold of Delsarte type in $\PP^4(Q)$ with matrix~$A$. If there exists $\boldsymbol{\alpha}'$ with $\mathcal{A}_H(\boldsymbol{\alpha}' )<f$, then $\mathcal{A}_H(\boldsymbol{\alpha}' )=\mathcal{A}_H(\boldsymbol{\alpha}_0 )$; furthermore, the height of $\Phi_{\xtil_A}$ is finite and equal to the length $($cardinality$)$ of the $H$-orbit of $\boldsymbol{\alpha}_0$.
\end{Lemma}

\begin{proof}Suppose that $\| t\boldsymbol{\alpha}' \| \geq 1$ for all $t\in H$. Then as $\# H=f$, we have $\mathcal{A}_H(\boldsymbol{\alpha}' )\geq f$; but, this is against the assumption. Hence $\| t\boldsymbol{\alpha}' \| =0$ for some $t\in H$ and by the uniqueness of $\boldsymbol{\alpha}_0$, we find $t\boldsymbol{\alpha}'=\boldsymbol{\alpha}_0$. Since~$H$ is a subgroup of~$(\ZZ /d\ZZ )^{\times}$, we find $\mathcal{A}_H(\boldsymbol{\alpha}' )
=\mathcal{A}_H(\boldsymbol{\alpha}_0 )$ and this holds for every element in the $H$-orbit of $\boldsymbol{\alpha}_0$. Therefore the number of $\boldsymbol{\alpha}$'s with $\mathcal{A}_H(\boldsymbol{\alpha} )<f$ is equal to the length of the $H$-orbit of $\boldsymbol{\alpha}_0$ and so is the height of~$\Phi_{\xtil_A}$.
\end{proof}

We are going to calculate the length of the $H$-orbit of $\boldsymbol{\alpha}_0$, namely the number of distinct elements in $\{ t\boldsymbol{\alpha}_0 \,|\, t\in H \}$. Given $\boldsymbol{\alpha}_0=(\alpha_0 , \alpha_1 ,\dots ,\alpha_4 )$ with $\| \boldsymbol{\alpha}_0 \| =0$, we may choose every $\alpha_i$ as $0<\alpha_i <d$ so that $\alpha_0+\alpha_1+\cdots +\alpha_4=d$; then we set
\begin{gather} \label{d}
e=\gcd (\alpha_0 , \alpha_1 ,\dots ,\alpha_4, d ) \qquad \text{and}\qquad d_A =\frac{d}{e}.
\end{gather}
We immediately see $e<d$ and $d_A\geq 2$. Further, suppose $d_A=2$. Then it follows from $0<\alpha_i <d$ that $\alpha_i =e$ for all $i$, which implies $5e=d=2e$. Since this is absurd, we find $d_A\geq 3$.

\begin{Lemma} \label{slope} Let $\xtil_A$ be a Calabi--Yau threefold of Delsarte type in $\PP^4(Q)$ with matrix $A$. Let $d_A$ be the integer defined in~\eqref{d}. Write $f$ $($resp.~$f_A)$ for the order of $p$ modulo $d$ $($resp.\ modulo~$d_A)$. Then the following assertions hold.
\begin{enumerate}\itemsep=0pt
\item[$(1)$] $\mathcal{A}_H(\boldsymbol{\alpha}_0 )=\frac{f}{f_A} \sum\limits_{i=0}^{f_A -1} \| p^i \boldsymbol{\alpha}_0 \|$.
\item[$(2)$] $\mathcal{A}_H(\boldsymbol{\alpha}_0 )<f$ if and only if $\| p^i \boldsymbol{\alpha}_0 \| \leq 1$ for $0\leq i<f_A$.
\end{enumerate}
\end{Lemma}

\begin{proof} (1) Since $e=\gcd (\gcd (\alpha_0 , \alpha_1 ,\dots ,\alpha_4), d )$, write $\gcd (\alpha_0 , \alpha_1 ,\dots ,\alpha_4)=eg$ for some $g$ with $\gcd (g,d_A)=1$. Then
\begin{align*}
p^i \boldsymbol{\alpha}_0 =\boldsymbol{\alpha}_0 & \Leftrightarrow p^i \alpha_j \equiv \alpha_j \pmod{d} \\
 & \Leftrightarrow d\,|\, \big(p^i-1\big)\gcd (\alpha_0 , \alpha_1 ,\dots ,\alpha_4) \\
 & \Leftrightarrow d_A \,|\, \big(p^i-1\big)g \\
 & \Leftrightarrow p^i\equiv 1\pmod{d_A}.
\end{align*}
Hence
\begin{gather*}
\mathcal{A}_H(\boldsymbol{\alpha}_0 )=\frac{f}{f_A} \big( \| \boldsymbol{\alpha}_0 \| +
\| p \boldsymbol{\alpha}_0 \| +\cdots +\big\| p^{f_A -1} \boldsymbol{\alpha}_0 \big\| \big) =
\frac{f}{f_A} \sum_{i=0}^{f_A -1} \big\| p^i \boldsymbol{\alpha}_0 \big\|.
\end{gather*}

(2) Let $H'= \{ p^i\ (\operatorname{mod} d)\,|\, 0\leq i< f_A \}$ and write $a$ (resp.~$b$) for the multiplicity of $1$ (resp.~$2$) among the $\| t\boldsymbol{\alpha}_0 \|$'s for $t\in H'$. Depending on whether $H'\ni -1$ or not, we divide the proof into two cases:

(i) If $H' \ni -1$, then $a+b+2=f_A$ and we find
\begin{gather*}
\sum_{i=0}^{f_A -1} \big\| p^i \boldsymbol{\alpha}_0 \big\| =a+2b+3 =f_A+b+1>f_A.
\end{gather*}
Hence by (1), $\mathcal{A}_H(\boldsymbol{\alpha}_0 ) >f$.

(ii) If $H \not\ni -1$, then $a+b+1=f_A$ and one sees{\samepage
\begin{gather*}
\frac{f}{f_A} \sum_{i=0}^{f_A -1} \big\| p^i \boldsymbol{\alpha}_0 \big\| = \frac{f}{f_A} (0+1+\cdots +1+2+\cdots +2) =\frac{f(f_A+b-1)}{f_A}.
\end{gather*}
Hence $\mathcal{A}_H(\boldsymbol{\alpha}_0 )<f$ is equivalent to $b-1<0$ (i.e., $b=0$).}

Therefore the inequality $\mathcal{A}_H(\boldsymbol{\alpha}_0 )<f$ holds if and only if the case (ii) occurs with $b=0$; this is the case where $\| \boldsymbol{\alpha}_0 \|=0$ and $\| p^i \boldsymbol{\alpha}_0 \|=1$ for all $i$ with $1\leq i<f_A$.
\end{proof}

\begin{Theorem} \label{height-f}Let $\xtil_A$ be a Calabi--Yau threefold of Delsarte type in $\PP^4(Q)$ with matrix~$A$. Let~$d_A$ be the integer defined in~\eqref{d} and~$f_A$ be the order of~$p$ modulo~$d_A$. Write $h:=\h \Phi_{\xtil_A}$. Then the following assertions hold.
\begin{enumerate}\itemsep=0pt
\item[$(1)$] $h$ is finite if and only if $\| p^i \boldsymbol{\alpha}_0 \| \leq 1$ for $0\leq i<f_A$.
\item[$(2)$] If $h$ is finite, then $h=f_A$.
\end{enumerate}
\end{Theorem}

\begin{proof} (1) Recall first that the height of $\Phi_{\xtil_A}$ can be computed from the Newton polygon of $F_d/\Gamma_A$ and its slope-less-than-$1$ part has length equal to the number of $\boldsymbol{\alpha} \in \mathfrak{A} (X_A)$ satisfying $\mathcal{A}_H(\boldsymbol{\alpha} )<f$. Write $K$ for the quotient field of the ring $W(k)$ of Witt vectors over $k$. Then by~\cite{AM},
\begin{align*}
h <\infty & \Leftrightarrow \dim_K \big(H^3_{\rm cris} \big(\xtil_A \big) \otimes K_{[0,1[}\big) \geq 1 \\
 & \Leftrightarrow \# \{\boldsymbol{\alpha} \in \mathfrak{A} (X_A)\, |\, \mathcal{A}_H(\boldsymbol{\alpha} )<f\} \geq 1 \\
 & \Leftrightarrow \mathcal{A}_H(\boldsymbol{\alpha}_0 )<f \\
 & \Leftrightarrow \big\| p^i \boldsymbol{\alpha}_0 \big\| \leq 1 \quad \text{for all $i$ with $0\leq i<f_A$.}
\end{align*}
The last equivalence follows from Lemma~\ref{slope}.

(2) Let $H$ be the orbit of $p$ in $(\ZZ /d\ZZ )^{\times}$. From the proof of Lemma~\ref{slope}, one sees that the length of the $H$-orbit of $\boldsymbol{\alpha}_0$ is equal to $f_A$. Hence together with Lemma \ref{lambda-0}, we have the following:
\begin{align*}
h & = \dim_K \big(H^3_{\rm cris} \big(\xtil_A \big)\otimes K_{[0,1[}\big) = \# \{\boldsymbol{\alpha} \in \mathfrak{A} (X_A)\, |\, \mathcal{A}_H(\boldsymbol{\alpha} )<f\} \\
 & = \text{the length of the $H$-orbit of $\boldsymbol{\alpha}_0$} =f_A.\tag*{\qed}
\end{align*}\renewcommand{\qed}{}
\end{proof}

For actual calculations, it is simpler to work modulo $d_A$ rather than modulo $d$. We thus reformulate Lemma~\ref{slope} and Theorem~\ref{height-f} in terms of modulo $d_A$.

\begin{Corollary} \label{alpha_A}With the assumptions as in Theorem {\rm \ref{height-f}}, define
\begin{gather*}
\boldsymbol{\alpha}_A=\frac{1}{e} \boldsymbol{\alpha}_0 =\left( \frac{\alpha_0}{e} ,\dots ,\frac{\alpha_4}{e} \right).
\end{gather*}
Let $H_A=\{ p^i \pmod{d_A} \,|\, 0\leq i<f_A \}$, and for $\boldsymbol{\beta} =(\beta_0,\dots ,\beta_4 )\in (\ZZ /d_A \ZZ )^5$, write
\begin{gather*}
\| \boldsymbol{\beta} \|_{d_A} = \sum_{i=0}^{4} \left\langle \frac{\beta_i}{d_A} \right\rangle -1, \qquad \mathcal{A}_{H_A}(\boldsymbol{\beta} )=\sum_{t\in H_A} \| t\boldsymbol{\beta} \|_{d_A}.
\end{gather*}
Then we may regard $\boldsymbol{\alpha}_A$ as an element of $(\ZZ /d_A \ZZ )^5$ and have the following.
\begin{enumerate}\itemsep=0pt
\item[$(1)$] $\| p^i \boldsymbol{\alpha}_0 \| =\| p^i \boldsymbol{\alpha}_A \|_{d_A}$.

\item[$(2)$] $\mathcal{A}_{H}(\boldsymbol{\alpha}_0 ) =\frac{f}{f_A} \mathcal{A}_{H_A}(\boldsymbol{\alpha}_A )$.

\item[$(3)$] $h$ is finite if and only if $\| p^i \boldsymbol{\alpha}_A \|_{d_A} \leq 1$ for $0\leq i<f_A$.
\end{enumerate}
\end{Corollary}

\begin{proof} (1) We write $\boldsymbol{\alpha}_A=(\beta_0,\dots, \beta_4 )$. Then
\begin{gather*}
\big\| p^i \boldsymbol{\alpha}_0 \big\| = \sum_{j=0}^{4} \left\langle \frac{p^i\alpha_j}{d} \right\rangle -1 =\sum_{j=0}^{4} \left\langle \frac{p^ie\beta_j}{ed_A} \right\rangle -1 =
\sum_{j=0}^{4} \left\langle \frac{p^i\beta_j}{d_A} \right\rangle -1 =\big\| p^i \boldsymbol{\alpha}_A \big\|_{d_A}.
\end{gather*}

(2) This follows from Lemma~\ref{slope}.

(3) This follows from (1) and Theorem~\ref{height-f}.
\end{proof}

\begin{Corollary} \label{supersing} With the assumptions and notation as in Theorem~{\rm \ref{height-f}} and Corollary~{\rm \ref{alpha_A}}, we have the following.
\begin{enumerate}\itemsep=0pt
\item[$(1)$] If $p\equiv 1\pmod{d_A}$, then $h=1$.
\item[$(2)$] If there exists $\mu$ such that $p^{\mu}\equiv -1\pmod{d_A}$, then $h$ is infinite.
\end{enumerate}
\end{Corollary}

\begin{proof} (1) If $p\equiv 1\pmod{d_A}$, then $f_A=1$ and $\| \boldsymbol{\alpha}_A \|_{d_A} =0\leq 1$. By Theorem~\ref{height-f}, $h=1$.

(2) $p^{\mu}$ is in $H_A$ and we find $\| p^{\mu} \boldsymbol{\alpha}_A \|_{d_A}= \| -\boldsymbol{\alpha}_A \|_{d_A} =3-\| \boldsymbol{\alpha}_A \|_{d_A}=3$. Hence by Theorem~\ref{height-f}, $h$ is infinite.
\end{proof}

\looseness=-1 Compared with the formal groups of $K3$ surfaces of Delsarte type in~\cite{G1}, it seems rather restrictive to have $\| p^i \boldsymbol{\alpha}_A \|_{d_A}=1$ for all $1\leq i<f_A$. This may be a reason why the infinite height occurs more often than finite height for Calabi--Yau threefolds~$\xtil_A$; see examples in Section~\ref{section5}.

\begin{Remark}For $K3$ surfaces of Delsarte type in $\PP^3(Q)$, an analogous statement to Corollary~\ref{supersing} holds, where the converse to (2) is also true (cf.~\cite{G1}). But, for a threefold~$\xtil_A$, the converse to~(2) does not hold in general. In fact, Example \ref{ex5-4} below shows that the orbit of $p\equiv 5\pmod{8}$ does not contain~$-1$, while $h$ is infinite.
\end{Remark}

\section{Calabi--Yau threefolds of weighted Fermat type}\label{section5}

In this section, we apply the results of the previous section to weighted Fermat threefolds and compute the height of the formal group of a crepant resolution $\xtil_A$. Let $X_A$ be a weighted Fermat threefold defined by the equation
\begin{gather*}
X_A\colon \ x_0^{d_0} +x_1^{d_1} +x_2^{d_2} +x_3^{d_3} +x_4^{d_4} =0
\end{gather*}
with $d_iq_i=m$ for $0\leq i\leq 4$. When $q_0+q_1+q_2+q_3+q_4 =m$, a crepant resolution $\xtil_A$ of $X_A$ is Calabi--Yau. Here Yui~\cite{Y2} has observed that there are 147 possibilities for $Q=(q_0,\dots ,q_4)$. First we restate Theorem~\ref{height-f} for weighted Fermat threefolds.

\begin{Proposition}Let $\xtil_A$ be a Calabi--Yau threefold of Fermat type in $\PP^4(Q)$ of degree $m$ with matrix $A$. Then $\boldsymbol{\alpha}_A =(q_0,q_1,q_2,q_3,q_4)$ and $d_A=m$. Furthermore, if $h=\h \Phi_{\xtil_A}$ is finite, then~$h$ is equal to the order of~$p$ modulo~$m$.
\end{Proposition}

\begin{proof}For weighted Fermat threefolds, we find that $d=\det A =d_0d_1d_2d_3d_4$ and $\boldsymbol{\alpha}_0 =(d_1d_2d_3d_4, \dots ,d_0d_1d_2d_3 )$ from the definition of $\mathfrak{A} (X_A)$. Since $d_iq_i=m$ and $\gcd (q_0,\dots ,q_4) =1$, we see $\boldsymbol{\alpha}_A =(q_0,q_1,q_2,q_3,q_4)$ and $d_A=m$. The rest of the claim follows from Theorem~\ref{height-f}.
\end{proof}

Following are some results obtained from our calculations.

\begin{Proposition} \label{h-d} Let $\xtil_A$ be a Calabi--Yau threefold of weighted Fermat type. Then the fol\-lo\-wing is a complete list of possible finite values for the height of the formal group of $\xtil_A$, where ``possible'' means that the values appear for some $\xtil_A$ in some characteristic~$p$:
\begin{gather*}
1, 2, 3, 4, 5, 6, 7, 8, 9, 10, 11, 12, 14, 16, 18, 20, 21, 22, 42.
\end{gather*}
\end{Proposition}

\begin{Example}Let $m=5$ and $Q=(1,1,1,1,1)$. Let $X_A$ be the weighted Fermat threefold defined by $x_0^{5} +x_1^{5} +x_2^{5} +x_3^{5} +x_4^{5} =0$. Assume $p\neq 5$. Then we have $h^{1,1}=1$, $h^{1,2}=101$, $d_A=5$ and
\begin{gather*}
h=\begin{cases}
1 & \text{if } p\equiv 1\pmod{5}, \\
\infty & \text{otherwise. }
\end{cases}
\end{gather*}
\end{Example}

\begin{Example} \label{ex5-4} Let $m=8$ and $Q=(1,1,1,1,4)$. Let $X_A$ be the weighted Fermat threefold defined by $x_0^{8} +x_1^{8} +x_2^{8} +x_3^{8} +x_4^{2} =0$. Assume $p\neq 2$. Then we have $h^{1,1}=1$, $h^{1,2}=149$, $d_A=8$ and
\begin{gather*}
h=\begin{cases}
1 & \text{if } p\equiv 1\pmod{8}, \\
2 & \text{if } p\equiv 3\pmod{8}, \\
\infty & \text{if } p\equiv 5, 7\pmod{8}.
\end{cases}
\end{gather*}
\end{Example}

\begin{Example}Let $m=966$ and $Q=(2,21,138,322,483)$. Let $X_A$ be the weighted Fermat threefold defined by $x_0^{483} +x_1^{46} +x_2^{7} +x_3^{3} + x_4^{2} =0$. Assume $p\nmid 966$. Then we have $h^{1,1}=h^{1,2}=143$, and $d_A=966$. For instance, if $p\equiv 43\pmod{966}$, then $h=22$.
\end{Example}

\begin{Example} Let $m=1806$ and $Q=(1,42,258,602,903)$. (This is the largest degree for the Fermat type.) Let $X_A$ be the weighted Fermat threefold defined by $x_0^{1806} +x_1^{43} +x_2^{7} +x_3^{3} +x_4^{2} =0$. Assume $p\nmid 1806$. Then we have $h^{1,1}=h^{1,2}=251$, $d_A=1806$ and
\begin{gather*}
h=\begin{cases}
1 & \text{if } p\equiv 1\pmod{m} \qquad \text{(1 class)}, \\
2 & \text{if } p\equiv 85,\dots \pmod{m} \qquad \text{(3 classes)}, \\
3 & \text{if } p\equiv 79,\dots \pmod{m} \qquad \text{(6 classes)}, \\
6 & \text{if } p\equiv 295,\dots \pmod{m} \qquad \text{(6 classes)}, \\
7 & \text{if } p\equiv 127,\dots \pmod{m} \qquad \text{(6 classes)}, \\
14 & \text{if } p\equiv 211,\dots \pmod{m} \qquad \text{(6 classes)}, \\
21 & \text{if } p\equiv 169,\dots \pmod{m} \qquad \text{(12 classes)}, \\
42 & \text{if } p\equiv 421,\dots \pmod{m} \qquad \text{(12 classes)}, \\
\infty & \text{otherwise } \qquad \text{(452 classes).}
\end{cases}
\end{gather*}
\end{Example}

\section{More examples}\label{section6}

Here we consider another type of polynomials. Let $X_A$ be a weighted Delsarte threefold defined by the equation:
\begin{gather*}
x_0^{m_0}x_1 +x_1^{m_1}+x_2^{m_2}+x_3^{m_3}+x_4^{m_4}=0 \subset \PP^4(Q).
\end{gather*}
It has degree $m=q_0m_0+q_1=m_1q_1=\cdots =m_4q_4$ and it may be called a {\it weighted quasi-diagonal threefold of degree $m$} (cf.~\cite{Y2}). There are $137$ weights to realize weighted quasi-diagonal threefolds. Several quantities associated with them are computed as follows.

\begin{Proposition}Let $X_A$ be a weighted quasi-diagonal threefold in $\PP^4(Q)$ defined above. Let~$\xtil_A$ be a Calabi--Yau threefold given as a crepant resolution of $X_A$. With the notation as in Corollary~{\rm \ref{alpha_A}}, set
\begin{gather*}
M=\operatorname{lcm} (m_0,m_2,m_3,m_4),
\end{gather*}
$M_i=M/m_i$, $i=0,2,3,4$, and $M_1=M-M_0-M_2-M_3-M_4$. Then
\begin{enumerate}\itemsep=0pt
\item[$(1)$] $d_A =M$ and $\boldsymbol{\alpha}_A =(M_0,M_1,M_2,M_3,M_4)$.
\item[$(2)$] Let $f_A$ be the order of $p$ modulo $M$. Then $h:=\h \Phi_{\xtil_A}$ is finite if and only if
$\| p^i \boldsymbol{\alpha}_A \|_{d_A}\leq 1$ for all~$i$, $0\leq i<f_A$.
\item[$(3)$] If $h$ is finite, then $h=f_A$.
\end{enumerate}
\end{Proposition}

\begin{proof}See Theorem 9.2 of \cite{GKY} for claim (1). The rest is a consequence of Theorem \ref{height-f}.
\end{proof}

\begin{Proposition} \label{h-qd} Let $\xtil_A$ be a Calabi--Yau threefold arising from weighted quasi-diagonal threefold. Then the following is a complete list of possible finite values for the height of the formal group of~$\xtil_A$, where ``possible'' means that the values appear in some characteristic~$p$:
\begin{gather*}
1, 2, 3, 4, 5, 6, 7, 8, 9, 10, 11, 12, 14, 15, 16, 18, 20,
21, 22, 23, 24, 27, 28, 30, 41, 42, 46, 82.
\end{gather*}
\end{Proposition}

\begin{Example} \label{ex6-3}
Let $m=84$ and $Q=(1,1,12,28,42)$. Let $X_A$ be the weighted quasi-diagonal
threefold defined by
\begin{gather*}
x_0^{83}x_1+x_1^{84}+x_2^{7}+x_3^{3}+x_4^{2}=0
\end{gather*}
in $\PP^4 (1,1,12,28,42)$. Assume $p\nmid 84$. Then $d_A =3486$, $h^{1,1}=11$,
$h^{2,1}=491$ and
\begin{gather*}
h= \begin{cases}
1 & \text{if } p\equiv 1 \pmod{3486}, \\
2 & \text{if } p\equiv 1163, 3319 \pmod{3486}, \\
 & \cdots \\
41 & \text{if } p\equiv 127, 169, 253, \dots \pmod{3486}, \\
82 & \text{if } p\equiv 43, 85, 211, \dots \pmod{3486}, \\
\infty & \text{otherwise.}
\end{cases}
\end{gather*}
\end{Example}

\begin{Remark}In Example \ref{ex6-3}, we find that $h > \min \big\{ h^{1,1}\big(\xtil_A \big)+1,h^{1,2}\big(\xtil_A \big)+1 \big\} =\min \{ 11+1$, $491+1 \}=12$ for some characteristic~$p$. According to Corollary~\ref{mirror3}, there is no group~$G$ of symplectic actions on~$X_A$ such that a crepant resolution of $Y=X_A/G$ becomes a mirror partner of~$\xtil_A$. (We note, however, that the construction in Corollary~\ref{mirror3} is very restrictive. A more general construction of mirror pairs in this direction is the Berglund--H\"ubsch--Krawitz mirror symmetry, where we may find a mirror partner of~$\xtil_A$.)
\end{Remark}

\begin{Remark} One may also compute the height {\rm $\h \Phi_{\xtil_A}$} for the following quasi-diagonal threefolds, but none of them gives a new value for the height beyond the lists of Propositions~\ref{h-d} and~\ref{h-qd}
\begin{gather*}
x_0^{m_0}x_2+x_1^{m_1}+x_2^{m_2}+x_3^{m_3}+x_4^{m_4}=0, \\
x_0^{m_0}+x_1^{m_1}x_2 +x_2^{m_2}+x_3^{m_3}+x_4^{m_4}=0, \\
x_0^{m_0}+x_1^{m_1} +x_2^{m_2}+x_3^{m_3}+x_3x_4^{m_4}=0.
\end{gather*}
From a view point of the Kreuzer--Skarke classification~\cite{KS} of invertible polynomials, these polynomials and the one in Proposition~\ref{h-qd} are of the same type (i.e., a chain of length $2$ and a~Fermat of length~$3$). Here we computed their height individually because the arithmetic properties of these polynomials are different. Since there are a few more polynomials of the same type, we should try to calculate the height for others as well.
\end{Remark}

\begin{Remark}There are also slightly more general types of weighted Delsarte threefolds. But, the height of {\rm $\Phi_{\xtil_A}$} for the following three cases is also within the lists of Propositions~\ref{h-d} and~\ref{h-qd}:
\begin{gather*}
x_0^{m_0}x_1 +x_0x_1^{m_1} +x_2^{m_2}+x_3^{m_3}+x_4^{m_4}=0, \\
x_0^{m_0}x_2 +x_1^{m_1} +x_0x_2^{m_2} +x_3^{m_3}+x_4^{m_4}=0, \\
x_0^{m_0}x_1 +x_1^{m_1}x_2 + x_0x_2^{m_2} +x_3^{m_3}+x_4^{m_4}=0.
\end{gather*}
We note that the first two of these polynomials are of the same type in the Kreuzer--Skarke classification. Since there are still other types of polynomials, it would be interesting to discuss the height of formal groups of weighted Delsarte threefolds from a view point of the Kreuzer--Skarke classification.
\end{Remark}

\vspace{-3mm}

\subsection*{Acknowledgements}

While the author was preparing for this manuscript, he visited Noriko Yui several times at the Department of Mathematics and Statistics of Queen's University and at the Fields Institute in Canada. He thanks Professor Yui for many inspiring discussions and is grateful to the two institutions for their hospitality. The author also thanks the Banff International Research Station in Canada for the workshop on Modular Forms in String Theory in 2016 where the main result of this paper was presented. Many thanks are due to the referees of the paper for useful comments and suggestions. This work was supported partially by the NSERC Discovery Grant of Noriko Yui at Queen's University in Canada and by the author's JSPS Grant-in-Aid for Scientific Research (C) 15540001 and 15K04771.

\vspace{-2.5mm}

\pdfbookmark[1]{References}{ref}
\LastPageEnding

\end{document}